\documentclass[12pt]{article} 
  
\usepackage{amsthm, amsfonts}
\usepackage{url}
\usepackage{amscd,amsmath,amssymb}

\newtheorem{definition}{Definition}
\newtheorem{theorem}{Theorem}

\newtheorem{proposition}{Proposition}
\newtheorem{corollary}{Corollary}

\newtheorem{remark}{Remark}

\def\de{\delta}

\def\al{\alpha}

\def\La{\Lambda}
\def\la{\lambda}

\def\kappa{\varkappa}
\def\si{\sigma}

\def\C{{\mathbb C}}

\def\sn{{\mathfrak S}_n}
\def\sk{{\mathfrak S}_k}
\def\sinf{{\mathfrak S}_\infty}

\def\H{{\cal H}}
\def\K{{\cal K}}
\def\M{{\cal M}}
\def\Msym{{\cal M}^{\rm sym}}
\def\Mskew{{\cal M}^{\rm skew}}
\def\Toff{{\cal T}^{\rm off}}
\def\Poff{{\cal P}^{\rm off}}
\def\N{{\mathbb N}}

\def\beq{\begin{equation}}
\def\eeq{\end{equation}}
\def\bea{\begin{eqnarray*}}
\def\eea{\end{eqnarray*}}

\def\ch{\operatorname{ch}}
\def\Id{\operatorname{Id}}
\def\Ind{\operatorname{Ind}}

\def\Reg{\operatorname{Reg}}
\def\sgn{\operatorname{sgn}}

\urldef\vershik\url{vershik@pdmi.ras.ru}
\urldef\natalia\url{natalia@pdmi.ras.ru}
\urldef\pnikitin\url{pnikitin@pdmi.ras.ru}

\author{P.~P.~Nikitin\thanks{%
St.~Petersburg Department of Steklov Institute of Mathematics.
E-mail: \pnikitin.} \and
N.~V.~Tsilevich\thanks{%
St.~Petersburg Department of Steklov Institute of Mathematics and
St.~Petersburg State University.
E-mail: \natalia.}
\and A.~M.~Vershik\thanks{%
St.~Petersburg Department of Steklov Institute of Mathematics, 
St.~Petersburg State University, and Institute for Information Transmission Problems.
E-mail: \vershik.}}
\title{On the decomposition of tensor representations of symmetric groups\thanks{Supported by the RFBR grant 17-01-00433 (Sections~1--3, 5) and the RSF grant 17-71-20153 (Section~4, the infinite-dimensional case).}}

\date{}

\begin{document}
\maketitle

\begin{abstract}
Following the general idea of Schur--Weyl scheme and 
using two suitable symmetric groups (instead of one),
we try to make more explicit the classical problem of decomposing tensor representations of finite and infinite symmetric groups into irreducible components.
\end{abstract}

\section{Introduction}

The problem of decomposing representations of symmetric and other classical groups in spaces of tensors into irreducibles has a long history and is of great importance for applications. Nevertheless, to the authors' opinion, this topic, being entirely classical, is not sufficiently addressed  in textbooks, both old and new (see, e.g., \cite{Mur, Weyl, FH, Cech}).  

Our approach is as follows: extending the classical Schur--Weyl duality between the actions of the general linear group $\operatorname{GL}(n,\C)$ and the symmetric group $\sk$ in the space
 $(\Bbb C^n)^{\otimes k}$, we consider the relationship between representations of two symmetric groups, $\sk$ itself and the Weyl subgroup $\sn$ in
 $\operatorname{GL}(n,\C)$. This leads us to the introduction of the so-called \emph{decomposition tensor of tensor representations}.
    
In more detail,  we study the ``Schur--Weyl'' representation of the group $\sk$ in the space $\H=(\Bbb C^n)^{\otimes k}$ together with the commuting representation of the group $\sn\subset\operatorname{GL}(n,\C)$. Then the space $\H$ can be decomposed into $(\sn\otimes\sk)$-invariant subspaces in two different ways. 
The first decomposition, indexed by Young diagrams $\nu\vdash k$, is into the subspaces $\H_\nu$ of tensors of different symmetry types (signatures), e.g., symmetric or skew-symmetric. The second one, indexed by Young diagrams $\mu\vdash k$, is into the subspaces $\H_\mu$ of tensors with the same multiplicities of the multisets of indices. Thus we can consider the representations $\rho_{\mu,\nu}$ of $\sn$ in the intersections $\H_{\mu,\nu}$ of these subspaces,  and the main result of the paper, Theorem~\ref{th:main}, gives a formula for this representation.

Considering the multiplicities $a_{\mu,\nu}^\eta$, $\eta\vdash n$, $\mu,\nu\vdash k$, of irreducible representations $\pi_\eta$ of $\sn$  in this representation, it is easy to see that they are nonzero only for diagrams with at most $k$ cells in all rows except the first one. Moreover, denoting by $(n-|\la|,\la)$ the diagram with $n$ cells obtained from $\la\vdash l\le k$ by adding a row of length $n-|\la|$,
  the multiplicities $a_{\mu\nu}^{(n-|\la|,\la)}$ does not depend on $n$ for sufficiently large $n$ (and fixed $k$).
These stable values of coefficients determine what we have called the \emph{decomposition tensor of tensor representations}.\footnote{The authors have not been able to find this object in the literature, though it is very natural and even indispensable for the representation theory of symmetric groups. We also emphasize that the decomposition tensor under consideration is not the structure tensor of any algebra.} It is upper triangular: $\rho_{\mu,\nu}$ is nonzero if and only if $\nu\unrhd\mu$ in the sense of the natural (dominance) ordering on partitions.

The actual computation of the components of the decomposition tensor, i.e., the multiplicities $a_{\mu,\nu}^\eta$, is a difficult problem which hardly has a good (closed-form) answer, since it involves the computation of plethysm coefficients, which is well known to be a very hard problem. However, it is of interest to analyze the decomposition tensor and the representations $\rho_{\mu,\nu}$ for small dimensions, and we do this for $k\le4$.
 
The stability property mentioned above suggests to study the similar problem for tensor representations of the infinite symmetric group $\sinf$. The answer (Theorem~\ref{th:inf}) is also similar to that in the finite case, but, as often happens, the infinite case is simpler than the finite one, since the involved induced representations become irreducible. In fact, by the well-known Lieberman theorem, these are exactly the representations of $\sinf$ extendable to representations of the complete symmetric group ${\frak S}^\infty$.  We consider in more detail the case $\mu=(1^k)$ of purely off-diagonal tensors, which includes all these representations.

Finally, in the last section we give an interpretation of our results in terms of symmetric functions. Namely, we present an identity for symmetric functions that corresponds to the decomposition of the representation of $\sn$ in the Schur--Weyl space $(\Bbb C^n)^{\otimes k}$ into the sum $\oplus_{\mu,\nu\vdash k}\rho_{\mu,\nu}$ of representations  in the subspaces $\H_{\mu,\nu}$.

\section{The decomposition tensor}\label{sec:1}
For $n,k\in\N$, consider the Schur--Weyl space
$$
\H=(\C^n)^{\otimes k},
$$
and denote its natural basis by $e_{i_1\ldots i_k}=e_{i_1}\otimes\ldots\otimes e_{i_k}$, where $\{e_i\}_{i=1,\ldots,n}$ is the basis of $\C^n$. Given a partition $\la$ of a positive integer, we denote by $\pi_\la$ the corresponding irreducible representation of a symmetric group. For convenience, we denote by $\Id_i=\pi_{(i)}$ the identical representation of ${\mathfrak S}_i$.

In $\H$ we have commuting actions of $\sn$ and $\sk$. Namely, $\sn$ acts as a subgroup of $\operatorname{GL}(n,\C)$:
$$
g(e_{i_1}\otimes\ldots\otimes e_{i_k})=e_{g(i_1)}\otimes\ldots\otimes e_{g(i_k)},\qquad g\in\sn,
$$
and $\sk$ acts by permutations of factors:
$$
\si(e_{i_1}\otimes\ldots\otimes e_{i_k})=e_{i_{\si(1)}}\otimes\ldots\otimes e_{i_{\si(k)}},\qquad\si\in\sk.
$$

On the one hand, we have the natural Schur--Weyl  ``{\it symmetry type}'' decomposition into $(\sn\otimes\sk)$-invariant subspaces
\begin{equation}\label{eq:H_decomp_form}
\H=\sum_{\nu\vdash k}\H^{\rm SW}_\nu,
\end{equation}
where $\H^{\rm SW}_\nu$ is the isotypic component of the irreducible representation $\varrho_\nu$ of $\operatorname{GL}(n,\C)$ with signature $\nu$. Denoting by $\xi_\nu$ the representation of $\sn$ obtained by restricting $\varrho_\nu$ to $\sn\subset \operatorname{GL}(n,\C)$, we have that the representation of $\sn$ in $\H^{\rm SW}_\nu$ is isomorphic to $\dim\nu\cdot\xi_\nu$. A (not quite explicit) formula for the characteristics of $\xi_\nu$ is given in \cite{SchThib} (see also \cite[Ex.~7.74]{St}).

On the other hand, we have the ``{\it type of tensors}'' decomposition into $(\sn\otimes\sk)$-invariant subspaces
\beq\label{decomp2}
\H=\sum_{\mu\vdash k}\H^{\rm mult}_\mu,
\eeq
where $\H^{\rm mult}_\mu$ is the subspace spanned by all $e_{i_1\ldots i_k}$ such that the multiset of indices $\{i_1,\ldots,i_n\}$ is of type $\mu=(1^{m_1}2^{m_2}\ldots)$, i.\,e., has $m_j$ elements of multiplicity $j$ for $j=1,2,\ldots$. It is not difficult to see that the representation of $\sn$ in $\H^{\rm mult}_\mu$ is isomorphic to
\beq\label{formsrep}
\frac{k!}{\prod_i(i!)^{m_i}m_i!}\cdot \Ind_{{\mathfrak S}_l\times{\mathfrak S}_{n-l}}^{\sn}(\Reg_l\times\Id_{n-l}),
\eeq
where $l=\sum m_i=l(\mu)$ is the length of $\mu$ and $\Reg_l$ is the regular representation of~${\mathfrak S}_l$.

Thus we also have the ``double'' decomposition into $(\sn\otimes\sk)$-invariant subspaces
\beq\label{doubledec}
\H=\sum_{\mu,\nu\vdash k}\H_{\mu,\nu},
\eeq
where $\H_{\mu,\nu} = \H^{\rm mult}_\mu \cap \H^{\rm SW}_\nu$.  Denote by $\rho_{\mu,\nu}=\rho^n_{\mu,\nu}$ the representation of $\sn$ in $\H_{\mu,\nu}=\H^n_{\mu,\nu}$. A natural question is to find this representation.

\begin{remark}
{\rm We can restate the classical Schur--Weyl duality as follows: if we start with the commuting actions of $\sn$ and $\sk$ in $\H=(\C^n)^{\otimes k}$ and want to maximize the first factor preserving the commutation property, what we get is the action of $\operatorname{GL}(n,\C) \times \sk$ in $\H$. If we maximize the second factor instead, we get the action of $\sn \times \operatorname{Part}(k)$, where 
$\operatorname{Part}(k)$ is the {\it partition algebra} (see, for example, \cite{RamHalverson}). Thus we can regard the action under consideration also as a restriction of the action of $\sn \times \operatorname{Part}(k)$.
%
}
\end{remark}

Let $\mu=(1^{m_1}2^{m_2}\ldots)\vdash k$, i.\,e.,  $\sum im_i=k$, and denote by $l=\sum m_i$ the length (total number of nonzero parts) of $\mu$. We introduce the following representations labeled by collections of partitions $\La=(\la_1,\la_2,\ldots)$, where $\la_1\vdash m_1, \la_2\vdash m_2, \ldots$:
$$
R_{\La}=\Ind_{{\frak S}_{m_1}\times{\frak S}_{m_2}\times\ldots}^{{\mathfrak S}_l}(\pi_{\la_1}\times\pi_{\la_2}\times\ldots)
$$
and
$$
Q_{\La}=\Ind_{{\frak S}_{m_1}\times{\frak S}_{2m_2}\times\ldots\times{\frak S}_{im_i}\times\ldots}^{{\mathfrak S}_k}(\pi_{\la_1}[\Id_{1}]\times\pi_{\la_2}[\Id_{2}]\times\ldots\times\pi_{\la_i}[\Id_{i}]\times\ldots),
$$
where 
$$
\pi_{\la_i}[\Id_{i}]=\Ind_{{\frak S}_{i}\wr{\frak S}_{m_i}}^{{\frak S}_{im_i}}(\Id_i\wr\pi_{\la_i})
$$ 
is the representation of ${\frak S}_{im_i}$ induced from the representation $\Id_i\wr\pi_{\la_i}$ of the wreath product ${\frak S}_{i}\wr{\frak S}_{m_i}\subset{\frak S}_{im_i}$.

Given $\nu\vdash k$, denote
\beq\label{Pi}
\Pi(\mu,\nu)=\dim\nu\sum_{\la_1\vdash m_1,\la_2\vdash m_2,\ldots}\langle\pi_\nu,Q_{\La}\rangle R_{\La},
\eeq
where $\langle\pi_\nu,Q_{\La}\rangle$ is the multiplicity of the irreducible representation $\pi_\nu$ in $Q_{\La}$. Thus $\Pi(\mu,\nu)$ is a representation of ${\frak S}_l$.
Now the representation $\rho_{\mu,\nu}$, which corresponds to type of tensors~$\mu$ and symmetry type~$\nu$, is essentially a representation induced from $\Pi(\mu,\nu)$.

\begin{theorem}\label{th:main}
Given $\mu,\nu\vdash k$, the representation $\rho_{\mu,\nu}$ of $\sn$ in the space $\H_{\mu,\nu}$ of tensors of type~$\mu$ and symmetry~$\nu$ is given by the formula
\beq\label{main}
\rho_{\mu,\nu} = \Ind_{{\frak S}_l\times{\frak S}_{n-l}}^{\sn}(\Pi(\mu,\nu)\times\Id_{n-l}),
\eeq
where $l$ is the length of $\mu$.
\end{theorem}

\begin{proof}
For simplicity, first assume that $\mu=(p^q)$, so $k=pq$, $l=q$, and in~\eqref{Pi} we have $\La=\la\vdash q$, $R_\La=\pi_\la$, $Q_\La=\pi_\la[\Id_p]$.

Fix the natural embedding $({\mathfrak S}_p)^q\hookrightarrow\sk$. The normalizer of $({\mathfrak S}_p)^q$ in $\sk$ is exactly the wreath product ${\frak S}_{p}\wr{\frak S}_{q}$. Observe that an element of ${\frak S}_{p}\wr{\frak S}_{q}$ can be identified with a tuple $(g_1,\ldots,g_q,\si)$, where $g_1,\ldots,g_q\in{\mathfrak S}_p$, $\si\in{\mathfrak S}_q$.

Now consider the space
$$
\K=\{f:\sk\to\C[{\mathfrak S}_q]: \forall  g=(g_1,\ldots,g_q,\si)\in {\frak S}_{p}\wr{\frak S}_{q}, f(gh)=\Reg_q^{\mbox{\scriptsize right}}(\si)f(h)\},
$$
where $\Reg_q^{\mbox{\scriptsize right}}$ is the right regular representation of ${\mathfrak S}_q$ in $\C[{\mathfrak S}_q]$. Thus $\K$ is the space of the representation
$$
\Pi^{\mbox{\scriptsize right}}=\Ind_{{\frak S}_{p}\wr{\frak S}_{q}}^{\sk}(\Id_p\wr\Reg_q^{\mbox{\scriptsize right}}).
$$

On the other hand, there is also a representation $\Pi^{\mbox{\scriptsize left}}$ of ${\frak S}_{q}$ in $\K$, given by the formula 
$$
\Pi^{\mbox{\scriptsize left}}(\tau) f(h)=\Reg_q^{\mbox{\scriptsize left}}(\tau)f(h),\qquad\tau\in{\frak S}_{q},
$$
where $\Reg_q^{\mbox{\scriptsize left}}$ is the left regular representation of ${\mathfrak S}_q$ in $\C[{\mathfrak S}_q]$,
and it is not difficult to see that the representation of $\sn$ in $\H_\mu^{(2)}$ is isomorphic to
$$
\Pi^{(2)}_\mu=\Ind_{{\mathfrak S}_q\times{\mathfrak S}_{n-q}}^{\sn}(\Pi^{\mbox{\scriptsize left}}\times\Id_{n-q}).
$$

Now we have the decomposition
\begin{equation}\label{eq:CS_n_decompos}
\C[{\mathfrak S}_q]=\bigoplus_{\la\vdash q} H_\la^{\mbox{\scriptsize left}}\otimes H_\la^{\mbox{\scriptsize right}}=\bigoplus_{\la\vdash q}K_\la,
\end{equation}
where $H_\la^{\mbox{\scriptsize left}}$ and $H_\la^{\mbox{\scriptsize right}}$ are the spaces of the irreducible representation $\pi_\la$ for the left and right regular representations of ${\frak S}_{q}$, respectively, and $K_\la=H_\la^{\mbox{\scriptsize left}}\otimes H_\la^{\mbox{\scriptsize right}}$. It follows that $\K=\sum_{\la\vdash q}\K_\la$, where the subspaces $\K_\la$ of $\K$ consisting of the functions with values in $K_\la$ are $(\Pi^{\mbox{\scriptsize left}}\otimes\Pi^{\mbox{\scriptsize right}})$-invariant, and
$$
(\Pi^{\mbox{\scriptsize left}}\otimes\Pi^{\mbox{\scriptsize right}})|_{\K_\la}=\pi_\la\otimes \pi_\la[\Id_p]=R_\la\otimes Q_\la.
$$

By the Schur--Weyl duality, $\H=\sum_{\nu\vdash k}\H^{\rm SW}_\nu=\sum_{\nu\vdash k}(\varrho_\nu\times\pi_\nu)$. It follows that if $Q_\La=\sum_{\nu\vdash k}d_\nu\pi_\nu$, then the contribution to $\rho_{\mu,\nu}$ coming from $\K_\la$ is equal to
$$
\dim\pi_\nu\cdot d_\nu R_\la=\dim\nu\cdot \langle\pi_\nu, Q_\la\rangle R_\la,
$$
and~\eqref{main} follows.

It is easy to see that the desired result for an arbitrary partition $\mu\vdash k$ can be obtained by similar arguments with obvious modifications.
\end{proof}

\begin{proposition}[upper triangularity]\label{prop:upper-triang}
The representation $\rho_{\mu,\nu}$ is nonzero if and only if $\nu\unrhd\mu$ in the sense of the natural (dominance) ordering on partitions.
\end{proposition}

\begin{proof}
Consider the decomposition $\C[\sk]=\sum_{\nu\vdash k}K_\nu$ of the semisimple algebra $\C[\sk]$ into the direct sum of simple ideals
analogous to~\eqref{eq:CS_n_decompos}. 
Then the subspaces $\H^{\rm SW}_\nu$ in~\eqref{eq:H_decomp_form} can be written as $\H^{\rm SW}_\nu=\H\cdot K_\nu$,
and in the same way
\begin{equation}\label{eq:H_2_decomp}
\H_{\mu,\nu}=\H^{\rm mult}_\mu\cdot K_\nu.
\end{equation}
We use the following well-known description of $K_\nu$ in terms of Young symmetrizers $c_T$ (see \cite[Chap.~7, Ex.~9 and~18]{Fulton}):
\beq\label{K_nu}
K_\nu = \sum_{T\in\nu} c_T \C[\sn] = \sum_{T\in\nu} \C[\sn] c_T,
\eeq
where the sum is over all standard tableaux of shape $\nu$.

Now we will describe the subspace $\H^{\rm mult}_\mu$, for any $\mu$, as a space of  tabloids. For any tableau $T$, we denote the corresponding tabloid by $\{T\},$ and the row and column stabilizers by $R(T)$ and $C(T)$. It is easy to see that  $\H^{\rm mult}_\mu$ is isomorphic as a $\sn$-module to the space spanned by the pairs $(\{T\}, f)$ where $\{T\}$ is a tabloid of shape $\mu$ (we take one representative for each class of tabloids that differ only by a permutation of rows) and $f$ is a function from $\{1,2,\dots, n\}$ to itself such that 
$f(v) = f(w)$  if and only if there exists $\si\in R(T)$ such that $v = \si w$.

By~\eqref{eq:H_2_decomp} and~\eqref{K_nu}, we must check when  $\H_{\mu,\nu}   = \H^{\rm mult}_\mu\cdot \sum_{T\in\nu} c_T \C[\sn]=0$. We will show that if the condition $\nu\unrhd\mu$ is not satisfied, then $(\{T'\}, f) \cdot c_T = 0$ for any basis element $(\{T'\}, f)\in \H^{\rm mult}_\mu$. First, it obviously suffices to prove that $\{T'\} \cdot c_T = 0$ for any $\{T'\}$ of shape $\mu$. Second, we recall that $c_T = b_T \cdot a_T$, where
$a_T = \sum_{\si \in R(T)} \si$ and $b_T = \sum_{\si\in C(T)} \sgn(\si) \si$.
Now, by the definition of the dominance order, there exist two numbers that lie in the same row of $T'$ and in the same column of $T$, whence $\{T'\} \cdot b_T = 0$. 

Vice versa, if $\nu\unrhd\mu$, then there exist tabloids $\{T'\}$ such that $\{T'\} \cdot b_T \neq 0$ (for example, we can consider the ``natural filling'' in which the rows are filled successively from left to right and from top to bottom).
\end{proof}

Given $\eta\vdash n$, denote by $a_{\mu\nu}^\eta=\langle\pi_\eta,\rho_{\mu,\nu}\rangle$ the multiplicity of the irreducible representation $\pi_\eta$ of $\sn$ in $\rho_{\mu,\nu}$, so that
\beq\label{rho}
\rho_{\mu,\nu}=\sum_{\eta\vdash n}a_{\mu\nu}^\eta\pi_\eta.
\eeq
Also, denote by $\tilde\eta$ the diagram obtained from~$\eta$ by removing the first row. Conversely, given a diagram $\la=(\la_1,\la_2,\ldots)$ and $n\ge|\la|+\la_1$, denote by $(n-|\la|,\la)=(n-|\la|,\la_1,\la_2,\ldots)$ the diagram with $n$ cells obtained from $\la$ by adding a row of length $n-|\la|$. 

\begin{corollary}[stability]\label{cor:stable}
It follows from~\eqref{main} that $a_{\mu\nu}^\eta=0$ unless $|\tilde\eta|\le l=l(\mu)$. Moreover, for every diagram $\la$ with at most $l$ cells, the coefficient $a_{\mu\nu}^{(n-|\la|,\la)}$ does not depend on $n$ for sufficiently large $n$.
\end{corollary}

\begin{definition} 
Given $\mu,\nu\vdash k$ and $\la\vdash l$ with $0\le l\le k$,
denote by $T_{\mu\nu}^\la$ the stable value of $a_{\mu\nu}^{(n-|\la|,\la)}$ for sufficiently large $n$. We call $T_{\mu\nu}^\la$ the {\rm decomposition tensor of tensor representations}. 
\end{definition}

It follows from~\eqref{Pi} that $T_{\mu\nu}^\la$ is an integer multiple of $\dim\nu$, so it is convenient to write the decomposition tensor  $T_{\mu\nu}^\la$ as the {\it symbol}
\begin{equation}\label{eq:struct-tensor}
T_{\mu\nu} = \frac1{\dim\nu}\sum_{|\la| \le k} T_{\mu\nu}^\la\cdot\la.
\end{equation}
Having $T_{\mu\nu}$, we can recover the corresponding stable form for the decomposition of $\rho_{\mu,\nu}$ as
$$
\rho_{\mu,\nu} = \sum_{|\la| \le k} T_{\mu\nu}^\la \cdot \pi_{(n-|\la|,\la)}.
$$

\section{Examples}
We start with two obvious examples, just to illustrate our formulas.

\medskip\noindent{\bf Example 1 (diagonal tensors).} Let $\mu=(k)$ (i.\,e., we consider ``diagonal'' tensors of the form $\sum\al_ie_{ii\ldots i}$). Then $m_k=1$, $m_i=0$ for $i\ne k$, $l=1$, $Q_{(1)}=\pi_{(k)}$, $R_{(1)}=\pi_{(1)}$, and we have
$$
\Pi(\mu,\nu)=\dim\nu\langle\pi_\nu,\pi_{(k)}\rangle\pi_{(1)}=\begin{cases}\pi_{(1)},&\nu=(k),\\0,&\mbox{otherwise},\end{cases}
$$
so that (using Pieri's formula)
$$
\rho_{(k),(k)}=\Ind_{{\frak S}_1\times{\frak S}_{n-1}}^{\sn}(\pi_{(1)}\times\Id_{n-1})=\pi_{(n)}+\pi_{(n-1,1)},
$$
and $\rho_{(k),\nu}=0$ for $\nu\ne(k)$. Thus 
$T_{(k),(k)}=\emptyset+(1)$.

\medskip\noindent{\bf Example 2 (purely off-diagonal tensors).} Let $\mu=(1^k)$ (i.\,e., we consider tensors of the form $\sum\al_{i_1\ldots i_k}e_{i_1\ldots i_k}$ where the sum is over pairwise distinct $i_1,\ldots, i_k$). Then $m_1=k$, $m_i=0$ for $i\ne 1$, $l=k$,  $Q_\la=R_\la=\pi_\la$ for $\la\vdash k$, and we have
$$
\Pi((1^k),\nu)=\dim\nu\sum_{\la_1\vdash k}\langle\pi_\nu,\pi_{\la}\rangle\pi_{\la}=\dim\nu\cdot\pi_\nu,
$$
so that
$$
\rho_{(1^k),\nu}=\dim\nu\cdot\Ind_{{\frak S}_k\times{\frak S}_{n-k}}^{\sn}(\pi_{\nu}\times\Id_{n-k}).
$$
Again using Pieri's formula, we get
\begin{equation}\label{eq:Pieri's_rule}
\rho_{(1^k),\nu} = \pi_{(n-|\nu|,\nu)} +\sum_{\nu / \la\mbox{\scriptsize\ is a horizontal strip},\ \la\neq\nu} \pi_{(n-|\la|,\la)}.
\end{equation}
The sum~\eqref{eq:Pieri's_rule} has the ``highest'' term $\pi_{(n-|\nu|,\nu)}$
(as we will see in Sec.~\ref{sec:infinite}, it is the only term that survives as $n\to\infty$), all other terms being of the form $\pi_{(n-|\la|,\la)}$ with $\la$ strictly contained in $\nu$.

\medskip\noindent{\bf Example 3 (tensors of valence $k=2$).} 
From Examples~1,~2, we have (omitting the zero values and 
using~\eqref{eq:Pieri's_rule})
$$
\rho_{(2),(2)}=\pi_{(n)}+\pi_{(n-1,1)},
$$
$$
\rho_{(1^2),(2)} = \Ind_{{\frak S}_2\times{\frak S}_{n-2}}^{\sn}(\Id_2\times\Id_{n-2}) = \pi_{(n)}+\pi_{(n-1,1)}+\pi_{(n-2,2)},
$$
$$
\rho_{(1^2),(1^2)} = \Ind_{{\frak S}_2\times{\frak S}_{n-2}}^{\sn}(\pi_{(1^2)}\times\Id_{n-2}) = \pi_{(n-1,1)} +\pi_{(n-2,1^2)}.
$$

For clarity, we  describe the corresponding invariant subspaces (see~\cite{N}, and also \cite{Markov}, for details in the symmetric case for any $k$).  In this case, $\H^{\rm mult}_{(1^2)} = H_{(1^2),(2)} \oplus H_{(1^2),(1^2)}$ is the standard decomposition of the space of zero-diagonal matrices $\M=\{(a_{ij})\}$  into symmetric and skew-symmetric parts:  $H_{(1^2),(2)}=\Msym$, $H_{(1^2),(1^2)}=\Mskew$. Then it is easy to see that
$\Msym=\Msym_0\oplus\Msym_1\oplus\Msym_2$, where
\begin{eqnarray*}
\Msym_0&=&\{cE, c\in\C\}, \mbox{ where } E=(1-\de_{ij}),\\
\Msym_1&=&  \{(a_{ij}):  a_{ij}=\al_i+\al_j \mbox{ for } (\al_j)\in\C^n,\;\sum\al_j=0\},\\
\Msym_2&=&\{(a_{ij}):a_{ij}=a_{ji},\;\sum_ja_{ij}=0\mbox{ for every }i\}
\end{eqnarray*}
are the invariant subspaces corresponding to the irreducible representations $\pi_{(n)}$, $\pi_{(n-1,1)}$, and $\pi_{(n-2,2)}$, respectively. 

Similarly, we have $\Mskew=\Mskew_{(1)}\oplus\Mskew_{(1^2)}$, where
\begin{eqnarray*}
\Mskew_{(1)}&=&\{(a_{ij}): a_{ij}=\al_i-\al_j \mbox{ for } (\al_j)\in\mathbb R^n,\sum\al_j=0\}, \\
\Mskew_{(1^2)}&=&\{(a_{ij}): a_{ji}=-a_{ij},\;\sum_ja_{ij}=0 \mbox{ for every }i\}
\end{eqnarray*}
are the invariant subspaces corresponding to the irreducible representations $\pi_{(n-1,1)}$ and $\pi_{(n-2,1^2)}$, respectively. 

\medskip\noindent{\bf Example 4 (tensors of valence $k=3$).} Once again, from Examples~1,~2, we obtain
$$
\rho_{(3),(3)}=\pi_{(n)}+\pi_{(n-1,1)},
$$
$$
\rho_{(1^3),\nu}=\dim\nu\cdot\Ind_{{\frak S}_3\times{\frak S}_{n-3}}^{\sn}(\pi_{\nu}\times\Id_{n-3}),\qquad \nu\vdash 3.
$$
Again we use~\eqref{eq:Pieri's_rule} to rewrite the last formula in the form
\begin{gather*}
\rho_{(1^3),(3)}  		=  \pi_{(n)} + \pi_{(n-1,1)} + \pi_{(n-2,2)} + \pi_{(n-3,3)}, \\
\rho_{(1^3),(21)}  	=  2 \biggl( \pi_{(n-1,1)} + \pi_{(n-2,2)} + \pi_{(n-2,1^2)}+ \pi_{(n-3,2,1)} \biggr), \\
\rho_{(1^3),(1^3)}  		=  \pi_{(n-2,1^2)} + \pi_{(n-3,1^3)}.
\end{gather*}

Now let $\mu=(21)$. Then $m_1=m_2=1$, $m_i=0$ for $i\ne 1,2$, $Q_{(1),(1)}=\Ind_{{\frak S}_2\times{\frak S}_1}^{{\frak S}_3}(\Id_2\times\Id_1)$, $R_{(1),(1)}=\Ind_{{\frak S}_1\times{\frak S}_1}^{{\frak S}_2}(\Id_1\times\Id_1)$,
and we have
$$
\Pi((21),\nu)=\dim\nu\cdot\langle\pi_\nu,\pi_{(3)}+\pi_{(21)}\rangle\Reg_2,
$$
so that
$$
\rho_{(21),\nu}=\begin{cases}\Ind_{{\frak S}_2\times{\frak S}_{n-2}}^{\sn}(\Reg_2\times\Id_{n-2}),&\nu={(3)},\\
2\cdot\Ind_{{\frak S}_2\times{\frak S}_{n-2}}^{\sn}(\Reg_2\times\Id_{n-2}),&\nu={(21)},\\
0,&\nu=(1^3).
\end{cases}
$$

Thus we can summarize the information on the decomposition tensor $T_{\mu\nu}^\la$ for $k=3$ in the following Table~1. 

\begin{table}[!h]\label{table1}
\caption{$T_{\mu\nu}$ for $k=3$.}
{
\bigskip
\begin{tabular}{|c|c|c|c|}
\hline
$\mu\backslash\nu$&$(3)$&$(21)$&$(1^3)$\\\hline
$(3)$&$\emptyset+(1)$&&\\\hline
$(21)$&$\emptyset+2\cdot(1)+(2)+(1^2)$&$\emptyset+2\cdot(1)+(2)+(1^2)$&\\\hline
$(1^3)$&$\emptyset+(1)+(2)+(3)$&$(1)+(2)+(1^2)+(2,1)$ & $(1^2)+(1^3)$\\\hline

\hline
\end{tabular}}
\end{table}

\def\T{{\cal T}}
\def\Tsym{{\cal T}^{\rm sym}}
\def\Tskew{{\cal T}^{\rm skew}}

In this case, the structure of the invariant subspaces is also easy to describe. 
Here $\H^{\rm mult}_{(1^3)} = H_{(1^3),(3)} \oplus H_{(1^3),(2,1)} \oplus H_{(1^3),(1^3)}$ is the decomposition of the space of tensors $\T=\{(a_{ijk})\}$ with pairwise distinct indices into the symmetric part $H_{(1^3),(3)}=\Tsym$, skew-symmetric part $H_{(1^3),(1^3)}=\Tskew$, and the ``symmetry~$(2,1)$'' part $H_{(1^3),(2,1)}={\cal T}^{(2,1)}$ of the form
$$
{\cal T}^{(2,1)} = \{(a_{ijk}): (a_{ijk})\in {\cal T},\;\; a_{ijk}+a_{jki}+a_{kij} = 0 \mbox{ for any }i,j,k\}.
$$
As in the previous example,
$\Tsym=\Tsym_0 \oplus \Tsym_1 \oplus \Tsym_2 \oplus \Tsym_3$, where
\begin{eqnarray*}
\Tsym_0&=&\{(a_{ijk}): a_{ijk}= \al\mbox{ for } \al\in\C\},\\
\Tsym_1&=&  \{(a_{ijk}):  a_{ijk}=\al_i+\al_j+\al_k\mbox{ for } (\al_j)\in\C^n,\;\;\sum\al_j=0\},\\
\Tsym_2&=&\{(a_{ijk}): a_{ijk} = \al_{ij}+\al_{jk}+\al_{ki},\;\; (\al_{ij})\in \Msym_0\}, \\
\Tsym_3&=&\{(a_{ijk}): (a_{ijk})\in {\cal T}^{(3)},\;\;\sum_k a_{ijk}=0\mbox{ for any } i,j\},
\end{eqnarray*}
the invariant subspace $\Tsym_i$ corresponding to the irreducible representation $\pi_{(n-i,i)}$, $0\le i\le 3$. 

In the skew-symmetric case, we have $\Tskew=\Tskew_{(1^2)} \oplus\Tskew_{(1^3)}$, where
\begin{eqnarray*}
\Tskew_{(1^2)} 		&=& \{(a_{ijk}): a_{ijk}=\al_{ij}+\al_{jk}+\al_{ki} \mbox{ for } (\al_{ij})\in \Mskew_0\}, \\
\Tskew_{(1^3)} 	&=& \{(a_{ijk}): (a_{ijk})\in \Tskew,\;\sum_ka_{ijk}=0 \mbox{ for any }i,j\}
\end{eqnarray*}
are the invariant subspaces corresponding to the irreducible representations $\pi_{(n-2,1^2)}$ and $\pi_{(n-3,1^3)}$, respectively. 

The structure of the invariant subspaces for $\T^{(2,1)}$ is slightly more complicated; it will be more natural to use the decompositions~\eqref{eq:H_2_decomp} and~\eqref{K_nu}  and to describe the subspaces corresponding to particular Young symmetrizers. We have
$$
\T^{(2,1)} = \T^{(2,1)}c_{T} \oplus \T^{(2,1)}c_{T'},
$$
where $T, T'$ are the two standart tableau of shape $(2,1)$, so that
$$
c_T = (e + (12))\cdot(e - (13)),\quad c_{T'} = (e + (13))\cdot(e - (12)).
$$
The description of the subspaces is quite similar for $\T^{(2,1)}c_{T}$ and $\T^{(2,1)}c_{T'}$, so we give it for $\T^{(2,1)}c_{T}$:
$$
\T^{(2,1)}c_{T} = \T^{(2,1)}_{(1)} \oplus \T^{(2,1)}_{(2)} \oplus {\cal T}^{(2,1)}_{(1^2)} \oplus \T^{(2,1)}_{(2,1)}, 
$$
where
\begin{eqnarray*}
\T^{(2,1)}_{(1)} 			&=& \{(a_{ijk}):  a_{ijk}=\al_i-\al_k\mbox{ for } (\al_j)\in\C^n,\;\sum\al_j=0\},\\
\T^{(2,1)}_{(2)}			&=& \{(a_{ijk}): a_{ijk} = \al_{ij}-\al_{jk},\; (\al_{ij})\in \Msym_0\}, \\
\T^{(2,1)}_{(1^2)}		&=& \{(a_{ijk}): a_{ijk} = \al_{ij}+\al_{jk}-2\al_{ki},\; (\al_{ij})\in \Mskew_0\}, \\
\T^{(2,1)}_{(2,1)} 	&=& \biggl\{(a_{ijk}): (a_{ijk})\in {\cal M}^{(2,1)}c_T,\; \sum_k a_{ijk}=\sum_k a_{ikj}=0\mbox{ for any } i,j \biggr\},
\end{eqnarray*}
each subspace $\T^{(2,1)}_\la$ corresponding to the irreducible representation $\pi_{(n-|\la|,\la)}$.
In particular, we have a nice description for the invariant subspace corresponding to the primary component $2\pi_{(n-3,2,1)}$:
$$
\biggl\{(a_{ijk}): (a_{ijk})\in \T^{(2,1)},\; \sum_k a_{ijk}=
\sum_k a_{ikj}=0\mbox{ for any } i,j \biggr\},
$$
but for other primary subspaces, the description is more complicated.

As to the subspaces $H_{(2,1),\nu}$, it is easy to check that the subspace of $\H^{\rm mult}_{(2,1)}$ with any fixed order of indices (say, spanned by the basis vectors of the form $e_i \otimes e_i \otimes e_j$ with different $i,j$) is naturally isomorphic to  $\H^{\rm mult}_{(1^2)}$ (for example, in the above case, an isomorphism is given by $\sum_{i\neq j} a_{ij} e_i \otimes e_j \mapsto \sum_{i\neq j} a_{ij} e_i \otimes e_i \otimes e_j$). Thus we can easily deduce the structure of the subspaces of $H_{(2,1),\nu}$ corresponding to irreducible representations from that for $H^{\rm mult}_{(1^2)}$ (see Example~3), both in the symmetric and skew-symmetric cases.

\medskip\noindent{\bf Example 5 (tensors of valence $k=4$).} As in the previous examples, we have
$$
\rho_{(4),(4)}=\pi_{(n)}+\pi_{(n-1,1)},
$$
$$
\rho_{(1^4),\nu}=\dim\nu\cdot\Ind_{{\frak S}_4\times{\frak S}_{n-4}}^{\sn}(\pi_{\nu}\times\Id_{n-4}),\qquad \nu\vdash 4,
$$
$$
\rho_{(31),\nu}=\begin{cases}\Ind_{{\frak S}_2\times{\frak S}_{n-2}}^{\sn}(\Reg_2\times\Id_{n-2}),&\nu={(4)},\\
3\cdot\Ind_{{\frak S}_2\times{\frak S}_{n-2}}^{\sn}(\Reg_2\times\Id_{n-2}),&\nu={(31)},\\
\mbox{otherwise}.
\end{cases}
$$

For $\mu=(2^2)$, we have $m_2=2$, $m_i=0$ for $i\ne 2$, $l=2$ and
$$
\Pi((2^2),\nu)=\dim\nu\cdot\sum_{\la\vdash 2}\langle\pi_\nu,\pi_\la[\Id_2]\rangle\cdot\pi_\la.
$$
It is well known (see, e.g., \cite[Ex.~I.8.6]{Mac}) that
$$
\pi_{(r)}[\Id_2]=\sum_{\tau\vdash 2r,\,\tau\mbox{{\scriptsize\ is even}}}\pi_\tau
$$
and
$$ 
\pi_{(1^r)}[\Id_2]=\sum\pi_\tau'
$$
summed over partitions $\tau$ with Frobenius coordinates $(\al_1-1,\ldots,\al_p-1|\al_1,\ldots,\al_p)$ where $\al_1>\ldots>\al_p>0$ and $\al_1+\ldots+\al_p=r$.
Thus we have $\pi_{(2)}[\Id_2]=\pi_{(4)}+\pi_{(2^2)}$ and $\pi_{(1^2)}[\Id_2]=\pi_{(31)}$, so that we obtain
$$
\Pi((2^2),\nu)=\begin{cases}
\pi_{(2)},&\nu=(4),\\
2\pi_{(2)},&\nu=(2^2),\\
3\pi_{(1^2)},&\nu=(31),\\
0,&\mbox{otherwise},
\end{cases}
$$
and the corresponding formulas for $\rho_{(2^2),\nu}$ follow.

Finally, if $\mu=(21^2)$, then $m_1=2$, $m_2=1$, $m_i=0$ for $i\ne 1,2$, $l=3$; for $\la\vdash2$, we have $R_{\la,(1)}=\Ind_{{\frak S}_2\times{\frak S}_2}^{{\frak S}_4}(\pi_\la\times\Id_2)$ and $Q_{\la,(1)}=\Ind_{{\frak S}_2\times{\frak S}_2}^{{\frak S}_4}(\pi_\la[\Id_2]\times\Id_2)$
 and we have, using known formulas,
$$
\Pi((21^2),\nu)=\begin{cases}
\pi_{(3)}+\pi_{(21)},&\nu=(4),\\
3(\pi_{(3)}+2\pi_{(21)}+\pi_{(1^3)}),&\nu=(31),\\
2(\pi_{(3)}+\pi_{(21)}),&\nu=(2^2),\\
3(\pi_{(21)}+\pi_{(1^3)}),&\nu=(21^2),\\
0,&\nu=(1^4),
\end{cases}
$$
and the corresponding formulas for $\rho_{(21^2),\nu}$ follow.

Now we can summarize the information on the decomposition tensor $T_{\mu\nu}^\la$ for $k=4$ in the following Table~2.

\begin{table}[!h]
\label{table2}
\caption{$T_{\mu\nu}$ for $k=4$.}
{\tiny
\bigskip
\begin{tabular}{|c|c|c|c|c|c|}
\hline
$\mu\backslash\nu$&$(4)$&$(31)$&$(2^2)$&$(21^2)$&$(1^4)$\\\hline
$(4)$&$\emptyset+(1)$&&&&\\\hline
$(31)$&$\emptyset+2\cdot(1)+(2)+(1^2)$&$\emptyset+2\cdot(1)+(2)+(1^2)$&&&\\\hline
$(2^2)$&$\emptyset+(1)+(2)$&$(1)+(1^2)$&$\emptyset+(1)+(2)$&&\\\hline
$(21^2)$&$\emptyset+2\cdot(1)+2\cdot(2)$&$\emptyset+3\cdot(1)+3\cdot(2)+3\cdot(1^2)$&$\emptyset+2\cdot(1)+2\cdot(2)$&$(1)+(2)+2\cdot(1^2)$&\\
&$+(1^2)+(3)+(21)$&$+(3)+2\cdot(21)+(1^3)$&$+(1^2)+(3)+(21)$&$+(21)+(1^3)$&\\\hline
$(1^4)$&$\emptyset+(1)+(2)$&$(1)+(2)+(1^2)$&$(2)+(21)+(2^2)$&$(1^2)+(21)$&$(1^3)+(1^4)$\\
&$+(3)+(4)$&$+(3)+(21)+(31)$&&$+(1^3)+(21^2)$&\\
\hline
\end{tabular}}
\end{table}

\section{The infinite case}\label{sec:infinite}

Now, in the spirit of asymptotic representation theory,
it is natural to consider $n=\infty$. Namely, in this case we have 
$$
\H=(\ell_2)^{\otimes k},
$$
with the $k$th tensor power of the unitary action of the infinite symmetric group $\sinf$ (the inductive limit of the groups $\sn$ with the natural embeddings, i.e., the group of all finitely supported permutations of $\N$) in $\ell_2$. 

It is easy to see that in this case we can
reproduce the arguments and notation from Sec.~\ref{sec:1}  with obvious modifications. In particular, we have decomposition~\eqref{eq:H_decomp_form}, in which we now should regard $\H_\nu^{\operatorname{SW}}$ as the isotypic component of the irreducible representation $\pi_\nu$ of $\sk$, and decomposition~\eqref{decomp2} defined in the same way as in Sec.~\ref{sec:1}. Thus we obtain decomposition~\eqref{doubledec} and denote by $\rho^\infty_{\mu,\nu}$ the representation of $\sinf$ in $\H_{\mu,\nu}$, which corresponds to type of tensors~$\mu$ and symmetry type~$\nu$. It turns out that the structure of $\rho^\infty_{\mu,\nu}$ is the same
as in the finite case (see Theorem~\ref{th:main}), namely, it is essentially a representation induced from the same representation $\Pi(\mu,\nu)$ from~\eqref{Pi}.

\begin{theorem}\label{th:inf}
Given $\mu,\nu\vdash k$,
$$
\rho^\infty_{\mu,\nu}=\Ind_{{\frak S}_l\times\sinf[l]}^{\sinf}(\Pi(\mu,\nu)\times\Id),
$$
where $l=l(\mu)$ is the length of $\mu$, the representation $\Pi(\mu,\nu)$ of ${\frak S}_l$ is given by~\eqref{Pi}, and $\sinf[l]$ is the subgroup in $\sinf$ consisting of the permutations that fix the elements $1,\ldots,l$.
\end{theorem}

\begin{proof}
It is easy to see that for fixed $k\in\N$ and $\mu,\nu\vdash k$, the representations $\rho_{\mu,\nu}^{n}$ of $\sn$ in $\H_{\mu,\nu}^{n}$ form an inductive chain and $\rho^\infty_{\mu,\nu}=\lim_{n\to\infty}\rho_{\mu,\nu}^{n}$. Now the claim follows from Theorem~\ref{th:main} and the properties of induced representations.
\end{proof}

Let 
\beq\label{Pi_irr}
\Pi(\mu,\nu)=\bigoplus_{\la\vdash l} d^\la_{\mu,\nu} \cdot\pi_\la
\eeq
be the decomposition of $\Pi(\mu,\nu)$ into irreducible representations of ${\frak S}_l$. Then
$$
\rho^\infty_{\mu,\nu} = \bigoplus_{\la\vdash l} d^\la_{\mu,\nu} \cdot \Ind_{{\frak S}_l\times\sinf[l]}^{\sinf}(\pi_\la\times\Id)=
\bigoplus_{\la\vdash l}d^\la_{\mu,\nu}\cdot\pi^\infty_\la,
$$
where, in contrast to the case of finite $n$, the representation 
$$ 
\pi^\infty_\la = \Ind_{{\frak S}_l\times\sinf[l]}^{\sinf}(\pi_\la\times\Id)
$$
of $\sinf$ is irreducible (see \cite{Binder}, and also \cite{PAMQ}).
Thus in the infinite case the  decomposition tensor  is much simpler than in the finite case:
\begin{equation}\label{eq:struct-tensor_infty}
T^{\infty,\la}_{\mu\nu} = d^\la_{\mu,\nu}.
\end{equation}
While in the case of finite $n$, to obtain the coefficients $T_{\mu\nu}^\eta$ of the decomposition tensor, one should rewrite~\eqref{main} as
$$
\rho_{\mu,\nu} = \Ind_{{\frak S}_l\times{\frak S}_{n-l}}^{\sn}\Biggl(\biggl(\bigoplus_{\la\vdash l} d^\la_{\mu,\nu} \cdot\pi_\la \biggr) \times \Id_{n-l} \Biggr) = 
\bigoplus_{\la\vdash l} d^\la_{\mu,\nu} \cdot \Ind_{{\frak S}_l\times{\frak S}_{n-l}}^{\sn} \Biggl( \pi_\la \times \Id_{n-l} \Biggr), 
$$
and then use the Pieri rule to decompose the induced representations on the right-hand side into irreducible representations.



For tensors of type $\mu = (1^k)$,  in the infinite case we can give the following complete description.
We consider the subspace $\Toff_k = \H^{\rm mult}_{(1^k)}$ of purely off-diagonal tensors of valence $k$, and
we have the following decomposition of the corresponding representation $\Poff_k$ into irreducible components for the action of 
$\sinf \times \sk$:
\begin{equation}\label{eq:sinf_diff-indices_decompos}
\Poff_k = \sum_{\nu\vdash k} \pi^\infty_\nu \otimes \pi_\nu.
\end{equation}
Denote the corresponding irreducible subspaces by $\T^\nu$, so that
$
\Toff_k = \bigoplus_{\nu\vdash k} \T^\nu.
$
Thus the algebra generated by the operators of $\rho^\infty_{(1^k),\nu}$ is always a type~I factor and
$$
\rho^\infty_{(1^k),\nu} = \dim \nu \cdot \pi^\infty_\nu.
$$

\begin{remark}\label{rem:Schur-Weyl_diff-indices}
{\rm We see from~\eqref{eq:sinf_diff-indices_decompos} that on $\Toff_k$ we have an analog of the Schur--Weyl duality: the actions of $\sinf$ and $\sk$ commute and generate the commutants of each other. In particular, the commutant of the primary component $\dim \nu \cdot \rho^\infty_\nu$ coincides with the simple ideal $K_\nu$ (see the beginning of the proof of Proposition~\ref{prop:upper-triang}). In other words, we can characterize the subspaces $\T^\nu$ as follows:
$$
\T^\nu = \{ T \in \Toff_k : T\cdot (1-P_\nu) = 0 \},
$$
where $P_\nu \in \C[\sk]$ is the orthogonal projection onto $K_\nu$.
}
\end{remark}

%
%
\medskip\noindent{\bf Example 6.}
Let us consider the first nontrivial case, tensors of valence~$3$. For $k=3$, we have
\begin{gather*}
\Poff_3{\big|}_{\sinf} = \pi^\infty_{(3)} + 2\pi^\infty_{(2,1)}+ \pi^\infty_{(1^3)}, \\
\Toff_3 = \Tsym \oplus \T^{(2,1)} \oplus \Tskew.
\end{gather*}
The subspaces of symmetric tensors $T^{(3)}=\Tsym$ and skew-symmetric tensors $T^{(1^3)} = \Tskew$ are irreducible,
and the ``symmetry~$(2,1)$'' part $\T^{(2,1)}$, which has  the form
$$
{\cal T}^{(2,1)} = \{(a_{ijk}): (a_{ijk})\in \Toff_3,\;\; a_{ijk}+a_{jki}+a_{kij} = 0 \mbox{ for any }i,j,k\},
$$
corresponds to the primary component $2 \pi^\infty_{(2,1)}$ with the commutant equal to~$K_{(2,1)}$.

Comparing this with the case of finite $n$, we see that in the infinite case the ``highest'' subspaces $\Tsym_{(3)}$, $\T^{(2,1)}_{(2,1)}$, and~$\Tskew_{(1^3)}$ are dense in $\Tsym$, $\T^{(2,1)}$,
and~$\Tskew$, respectively. 

\begin{remark}
{\rm The action in $\H$ of the infinite symmetric group $\sinf$ can be extended to an action of the complete symmetric group ${\frak S}^\infty$, which is the group of all permutations of $\N$. Here the corresponding representation $\bar\rho_{\mu,\nu}$ is given by
$$
\bar\rho_{\mu,\nu}=\bigoplus_{\la\vdash l}d^\la_{\mu,\nu}\cdot\bar\pi^\infty_\la,
$$
where $d^\la_{\mu,\nu}$ are from~\eqref{Pi_irr} and 
$$
\bar\pi^\infty_\la=\Ind_{{\frak S}_l\times{\frak S}^\infty[l]}^{{\frak S}^\infty}(\pi_\la\times\Id)
$$
(with ${\frak S}^\infty[l]$ being the subgroup in ${\frak S}^\infty$ consisting of the permutations that fix the elements $1,\ldots,l$) are irreducible representations of ${\frak S}^\infty$ by the well-known Lieberman theorem \cite{Lib}.}
\end{remark}

\section{Symmetric functions formulation}\label{sec:symf}
Denote by $\chi_{\mu,\nu}$ the character of $\rho_{\mu,\nu}$, and let $\psi_{\mu,\nu}=\ch\chi_{\mu,\nu}$ be its image under the characteristic map. Since induction from Young subgroups correspond to multiplication of Schur functions and induction from wreath products correspond to plethysm, we obtain the following.  

\begin{corollary}
In the above notation,
\beq\label{insymf}
\psi_{\mu,\nu}=\dim\nu\sum_{\la_1\vdash m_1,\la_2\vdash m_2,\ldots}\langle s_\nu,s_{\la_1}[h_{1}]s_{\la_2}[h_{2}]\ldots\rangle s_{\la_1}s_{\la_2}\ldots \cdot h_{n-l},
\eeq
where $s_\la$ are Schur functions, $h_\mu$ are complete symmetric functions, and $f[g]$ denotes the plethysm of symmetric functions.
\end{corollary}

\begin{corollary}
For $\mu=(1^{m_1}2^{m_2}\ldots)\vdash k$, we have
\beq
\sum_{\la_1\vdash m_1,\la_2\vdash m_2,\ldots}\langle h_1^k,s_{\la_1}[h_{1}] s_{\la_2}[h_{2}]\ldots\rangle s_{\la_1}s_{\la_2}\ldots =\frac{k!}{\prod_{i}(i!)^{m_i}m_i!}h_1^{\sum m_i}.
\eeq
\end{corollary}

\begin{proof}
Take the sum of~\eqref{insymf} over $\nu\vdash k$ and use~\eqref{formsrep}.
\end{proof}

Now let us denote by $\Xi_{n,k}$  the characteristics of the action of $\sn$ in the whole space $(\C^n)^{\otimes k}$ and consider the generating function
\beq\label{F}
F=\sum_{n=0}^\infty\sum_{k=0}^\infty\frac1{k!}\Xi_{n,k}.
\eeq

On the one hand, using decomposition~\eqref{decomp2} and formula~\eqref{formsrep}, we have 
\bea
F&=&\sum_{n=0}^\infty\sum_{k=0}^\infty\frac1{k!}\sum_{\mu=(1^{m_1}2^{m_2}\ldots)\vdash k}\frac{k!}{\prod_i(i!)^{m_i}m_i!}\cdot h_1^{\sum m_i}h_{n-\sum m_i}\\
&=&\left(\sum_{n=0}^\infty h_n\right)\sum_{m_1,m_2,\ldots\ge0}\frac{1}{\prod_i(i!)^{m_i}m_i!}\cdot h_1^{\sum m_i}=
h\cdot\prod_{i=1}^\infty\sum_{m_i=0}^\infty\frac{(h_1^{m_i}/i!)}{m_i!}\\&=&
h\cdot \prod_{i=1}^\infty e^{h_1/i!}=h\cdot e^{h_1(e-1)},
\eea
where we have denoted $h=h_0+h_1+h_2+\ldots$.

On the other hand, by the Schur--Weyl duality we have
$$
\Xi_{n,k}=\sum_{\la\vdash k}\dim\la\cdot\ch\xi_\la^{(n)},
$$
where $\xi_\la^{(n)}$ is the representation of $\sn$ obtained by restricting to $\sn\subset \operatorname{GL}(n,\C)$  of the irreducible representation of $\operatorname{GL}(n,\C)$ with signature $\la$. By a  formula proved in \cite{SchThib} (see also \cite[Ex.~7.74]{St}), we have
$$
\ch\xi_\la^{(n)}=\sum_{\nu\vdash n}\langle s_\la,s_\nu[h]\rangle s_\nu,\quad
\mbox{whence}\quad \Xi_{n,k}=\sum_{\nu\vdash n}\langle h_1^k,s_\nu[h]\rangle s_\nu.
$$
Thus
\bea
F=\sum_{n=0}^\infty\sum_{k=0}^\infty\frac1{k!}\sum_{\nu\vdash n}\langle h_1^k,s_\nu[h]\rangle s_\nu=
\sum_{\nu}\langle e^{h_1},s_\nu[h]\rangle s_\nu,
\eea
where the last sum ranges over all partitions $\nu$ of nonnegative integers. So, the two decompositions~\eqref{eq:H_decomp_form} and~\eqref{decomp2} of the space $(\C^n)^{\otimes k}$ into invariant subspaces of $\sn\otimes\sk$ correspond to the following identity for symmetric functions:
\beq\label{id1}
h\cdot e^{h_1(e-1)}=\sum_{\nu}\langle e^{h_1},s_\nu[h]\rangle s_\nu.
\eeq

Now let us use the Cauchy identity to write the right-hand side of~\eqref{id1} as
$$
\bigg\langle e^{h_1(y)},\sum_{\nu} s_\nu[h](y) s_\nu(x)\bigg\rangle=
\bigg\langle e^{h_1(y)},\exp\bigg(\sum_{n=1}^\infty\frac{p_n(x)p_n[h](y)}{n}\bigg)\bigg\rangle.
$$
By the properties of plethysm, we have $p_n[h]=h[p_n]=\sum h_j[p_n]=\sum p_n[h_j]$, where $p_n[h_0]=1$,
so that the exponential factor in the right-hand side equals
\bea
\exp\bigg(\sum_{j=0}^\infty\sum_{n=1}^\infty\frac{p_n(x)p_n[h_j](y)}{n}\bigg)=
\prod_{j=0}^\infty\exp\bigg(\sum_{n=1}^\infty\frac{p_n(x)p_n[h_j](y)}{n}\bigg)\\
=h(x)\cdot\prod_{j=1}^\infty\sum_{\la_j} s_{\la_j}(x)s_{\la_j}[h_j](y)
=h(x)\cdot\sum_{\la_1,\la_2,\ldots}\prod_{j=1}^\infty s_{\la_j}[h_j](y)\prod_{j=0}^\infty s_{\la_j}(x).
\eea
Thus we have
$$
F=h\cdot\sum_{\la_1,\la_2,\ldots}\bigg\langle e^{h_1},\prod_{j=1}^\infty s_{\la_j}[h_j]\bigg\rangle \prod_{j=1}^\infty s_{\la_j}.
$$
Since
$$
e^{h_1}=\sum_{k=0}^\infty \frac{h_1^k}{k!}=\sum_{k=0}^\infty \frac{1}{k!}\sum_{\nu\vdash k}\dim\nu\cdot s_\nu,
$$
we see, keeping in mind~\eqref{F}, that 
\bea
\Xi_{n,k}=\sum_{\nu\vdash k}\dim\nu\cdot \sum_{\begin{smallmatrix}
\la_1,\la_2,\ldots:\\\sum j|\la_j|=k\end{smallmatrix}}\bigg\langle s_\nu,\prod_{j=1}^\infty s_{\la_j}[h_j]\bigg\rangle \prod_{j=1}^\infty s_{\la_j}\cdot h_{n-\sum|\la_j|}\\
=\sum_{\nu\vdash k}\sum_{\mu=(j^{m_j})\vdash k}\sum_{\la_1\vdash m_1,\la_2\vdash m_2,\ldots}\dim\nu\cdot\bigg\langle s_\nu,\prod_{j=1}^\infty s_{\la_j}[h_j]\bigg\rangle \prod_{j=1}^\infty s_{\la_j}\cdot h_{n-\sum|\la_j|},
\eea
and, comparing with~\eqref{insymf}, we see that this is exactly the decomposition corresponding to~\eqref{doubledec}.


\begin{thebibliography}{55}

\bibitem{Binder} M.~W.~Binder, Irreducible induced representations of ICC-groups, {\it Math. Ann.} {\bf 294} (1992), 37--47.

\bibitem{Cech} T.~Ceccherini-Silberstein, F.~Scarabotti, and F.~Tolli, {\it
Representation Theory of the Symmetric Groups. 
The Okounkov--Vershik Approach, Character Formulas, and Partition Algebras}, Cambridge Univ. Press, Cambridge, 2010.

\bibitem{Fulton} W.~Fulton, {\it Young Tableaux, with Applications to Representation Theory and Geometry}, Cambridge Univ. Press, 1997.

\bibitem{FH} W.~Fulton and J.~Harris, {\it Representation Theory. 
A First Course}, Springer-Verlag, New York, 1991.

\bibitem{Lib}A.~Lieberman, The structure of certain unitary representations of infinite
symmetric groups, {\it Trans. Amer. Math. Soc.} {\bf164} (1972), 189--198.

\bibitem{Mac} I.~G.~Macdonald, {\it Symmetric Functions and Hall
Polynomials}, Oxford Univ. Press, New York, 1995.

\bibitem{Mur} F.~D.~Murnaghan, {\it The Theory of Group Representations}, Dover, New York, 1963.

\bibitem{RamHalverson} T.~Halverson and A.~Ram, Partition algebras,
{\it European J. Combin.} {\bf 26} (2005), 869--921.

\bibitem{N} P.~P.~Nikitin, A realization of the irreducible representations of $S_n$ corresponding to 2-row diagrams in square-free symmetric multilinear forms, {\it J. Math. Sci. (N. Y.)} {\bf 129}, No.~2 (2005), 3796--3799.

\bibitem{SchThib} T.~Scharf and J.-Y.~Thibon, A Hopf algebra approach to inner plethysm, {\it Adv. Math.} {\bf 104} (1994), 30--58.

\bibitem{St} R.~P.~Stanley, {\it Enumerative Combinatorics}, Vol.~2, Cambridge Univ. Press, Cambridge, 1999.

\bibitem{Markov} N.~V.~Tsilevich and A.~M.~Vershik, Markov measures on Young tableaux and induced representations of the infinite symmetric group {\it Prob. Theory Appl.} {\bf 51}, No.~1 (2006), 211--223.

\bibitem{PAMQ} N.~V.~Tsilevich and A.~M.~Vershik, Induced representations of the infinite symmetric group, {\it Pure Appl. Math. Quart.} {\bf 3}, No.~4 (2007), 1005--1026.

\bibitem{Weyl} H.~Weyl, {\it The Classical Groups. Their Invariants and Representations}, Princeton Univ. Press, Princeton, N.J., 1939.
\end{thebibliography}
\end{document}